\documentclass[reqno]{amsart}
\usepackage{hyperref}
\usepackage{amsfonts,amssymb,amsmath,amsthm}
\usepackage{url}
\usepackage{enumerate}
\usepackage{graphics}
\usepackage{graphicx}

\newtheorem{thrm}{Theorem}[section]

\theoremstyle{definition}

\numberwithin{equation}{section}

\begin{document}
\title[elastodynamics model]{ Riemann type initial boundary value problem for  a system in elastodynamics }
\author[K. Divya Joseph and  P. A  Dinesh \\ ]
{Kayyunnapara  Divya Joseph$^1$ and  P. A  Dinesh$^2$, \\}

\address{Department of Mathematics,MS Ramaiah Institute of Technology
MSRIT P. O,Bangalore 560054, India.\\}

\email{divyakj@msrit.edu, dineshdpa@msrit.edu}
\thanks{$^{1,2}$Department of Mathematics,M.S. Ramaiah Institute of Technology,
MSRIT P. O,Bangalore 560054, India.
Email: $^1$divyakj@msrit.edu,+917259149455,  $^2$dineshdpa@msrit.edu,+91984548356 }


\begin{abstract}
This paper is a continuation of our previous paper \cite{d1} on  the initial boundary value problem for a nonconservative system appearing in elastodynamics in the space time domain $x>0,t>0$. There, the  initial and boundary data were assumed to  lie on the
 level sets of one of the Riemann invariants where as in this paper we  consider the general initial and boundary data of Riemann type, formulate the boundary value  problem based on the Riemann problem and   construct explicitly the solution.
\end{abstract}
\maketitle

\section{Introduction} \label{sect1}
In this article, we study a  system of two equations, appearing in elastodynamics, which was studied in our previous paper\cite{d1}  namely,
\begin{equation}
\begin{gathered}
\frac{\partial u}{\partial t }+ u \frac{\partial u}{\partial x} - \frac{\partial{\sigma}}{\partial x} =
0,\\
\frac{\partial{\sigma}}{\partial t} + u \frac{\partial{\sigma}}{\partial x} - k^2
\frac{\partial{u}}{\partial x} = 0.
\end{gathered}
\label{e1.1}
\end{equation}
Here, $u$ is the velocity, $\sigma$ is the stress and $k>0$
is the speed of propagation of the elastic waves and the system corresponds to the coupling of a dynamical law and a Hooke's law in a homogeneous medium in which density varies slightly in the neighborhood of a constant value.

Initial value problem for  the system \eqref{e1.1} 
is well studied by many authors \cite{c1,c2,j2}, however the boundary value problem is not understood well.  In this paper, we study initial boundary value problem for \eqref{e1.1}, in $x>0,t>0$ with initial 
conditions
\begin{equation}
( u(x,0),\sigma (x, 0) ) = (u_0(x) , \sigma _0(x) ),\,\,\,
x > 0
\label{e1.2}
\end{equation}
and a weak form of the  boundary conditions,
\begin{equation}
( u(0,t),\sigma (0,t) ) = (u_b(t) , \sigma _b(t)
),\,\,\,t>0.
\label{e1.3}
\end{equation}

When  the initial and boundary data were assumed to  lie on the
 level sets of one of the Riemann invariants, we used the vanishing viscosity method to construct solution in \cite{d1}. The aim of the present paper is to  study  the general initial and boundary data of Riemann type and formulate the boundary value 
problem based on the Riemann problem  and construct an explicit solution. Solution of Riemann type data is fundamental because, this is used in the construction of solutions of  initial boundary value problem for the  general initial and bounday data using Glimms scheme and various difference approximations .

We write the system \eqref{e1.1} in the form,
\begin{equation*}
\begin{bmatrix}
u\\  
\sigma
\end{bmatrix} 
_t +
A(u,\sigma)
\begin{bmatrix}
u \\
\sigma
\end{bmatrix} 
_x
=0
\end{equation*}
with the matrix corresponding to the system,
\begin{equation*}
A(u,\sigma) =
\begin{bmatrix}
u & -1\\
-k^2 & u
\end{bmatrix}.
\end{equation*}

The eigenvalues of $A(u,\sigma)$ are called the
characteristic speeds of system \eqref{e1.1}. A simple
computation
shows that, the equation for the eigenvalues are given
by
\begin{equation*}
(\lambda - u )^2  - k^2 =0
\end{equation*}
which has two real distinct roots,
\begin{equation*}
\lambda_1 = u-k ,\,\,\,\lambda_2 = u + k.
\label{e1.4}
\end{equation*} 
So the system is strictly hyperbolic. 
The curves defined by $\frac{dx}{dt}= \lambda_1$ and  $\frac{dx}{dt}= \lambda_2$ are called  the $1-$characteristic curve and $2-$characteristic curve respectively. 
Call $\ r_1 $ and $\ r_2 $ right eigenvectors
corresponding to $ \lambda_1 $ and $ \lambda_2 $
respectively. An
 easy computation shows
 \begin{equation*}
r_1 = r_1(u,\sigma)=(1,k),\,\,\,r_2= r_2 (u,
\sigma)=(1,-k)
\end{equation*}
We say $w_i :R^2 \rightarrow R$ is a  $i$-Riemann
invariant if 
\begin{equation*}
Dw_{i}(u,v) . r_i(u,\sigma)=0, 
\end{equation*}
where $ D w_i(u,v)=(w_{iu},w_{i \sigma})$, is the  the gradient of $w_i$ w.r.t. $(u,\sigma)$.  Explicit global Riemann invariants are given by 
\[
w_1(u,\sigma) =\sigma - k u, \,\,\
w_2(u,\sigma)=\sigma+  k u.
\]
 Now consider the case 
where $u_0$, $\sigma_0$,  $u_b$, $\sigma_b$ are all constants.
 The equation \eqref{e1.1} with the initial and boundary conditions \eqref{e1.2}-\eqref{e1.3}
does not have a solution in general for the following reason. Let
$w_1$ denotes the 1 - Riemann invariant, and  $w_2$ denotes the
2 - Riemann invariant. Then for smooth solutions, \eqref{e1.1} is equivalent to the system
\begin{equation*}
\begin{gathered}
\frac{\partial w_1}{ \partial t} + (u+k) \frac{
\partial w_1}{ \partial x} = 0, \\
\frac{\partial w_2 }{\partial t} + (u-k) \frac{
\partial w_2 }{ \partial x} = 0.
\end{gathered}
\label{e2.4}
\end{equation*}
The initial and boundary conditions  \eqref{e1.2} and  \eqref{e1.3} become
\begin{equation*}
\begin{aligned}
w_1(x,0)&=w_{10}(x)=\sigma_0 - k u_0, \\
w_2(x,0)&=w_{20}(x)=\sigma_0 +  k u_0, \\
w_1(0,t)&=w_{1b}(t)=\sigma_b - k u_b, \\
w_2(0,t)&=w_{2b}(t)=\sigma_b +  k u_b. \\
\end{aligned}
\label{e2.5}
\end{equation*}
Now let $\beta_1$ be a 1- characteristic curve,
$\frac{d x}{d t} = u - k$ and $\beta_2$ be a 2- characteristic curve,
$\frac{d x}{d t} = u + k$
then we have seen in \cite{d1} that 
 $w_2$  is constant along the 1-characteristic curve and
$w_1$ is constant along $2$ characteristc curve. There are three cases.

Case 1: $u(0,t)+k <0$. Then $u(0,t)-k <0$ and hence,  no condition required on $x = 0$

Case 2:  $ u(0,t)- k < 0 $ and $ u(0,t) + k >0 $. Here, only one condition is 
required.
Case 3: $ u - k > 0 $ , then $ u + k >0 $ and so, both $\sigma(0,t)$ and $u(0,t)$ must be prescribed.
Since the speeds $\lambda_1=u-k,\lambda_2=u+k$, depend on the unknown $u$, we do not know appriori
the sign of $\lambda_i, i=1,2$ at the boundary $x=0$. So we cannot proceed as in the linear case.

Classically, we need to take into account the small scale effects introduced by vanishing viscosity approximation as in \cite{jl} for hyperbolic systems of conservation laws
or as  in \cite{d1} for the special case where the initial and boundary data are in one of the level sets of Riemann invariants. 
The system \eqref{e1.1} is nonconservative, strictly
hyperbolic system with genuinely nonlinear characteristic fields. It is well known that smooth global in time
solutions does not exist even if the initial data
is smooth. Since the system is nonconservative any
discussion of well-posedness of solution should be based
on a given nonconservative product in addition to
admissibility criterion for shock discontinuities.
 In this paper, we take 
the Volpert nonconservative product and Lax admissibility
conditions for shocks.
 First we recall
some known facts about the Riemann problem for
\eqref{e1.1}, see \cite{j2}. Here, the initial data takes the form
\begin{equation*}
(u(x,0),\sigma(x,0)) = \begin{cases}
(u_{-},\sigma_{-}),&\text{if } x< 0,\\
 (u_{+},\sigma_{+}),&\text{if } x>0. 
\end{cases}
\label{e2.6}
\end{equation*}
A shock wave is a weak solution of \eqref{e1.1}, with
speed $s$ is of the form

\begin{equation*}
(u(x,0),\sigma(x,0)) = \begin{cases}
(u_{-},\sigma_{-}),&\text{if } x < s t,\\
 (u_{+},\sigma_{+}),&\text{if } x > s t. 
\end{cases}
\label{e2.7}
\end{equation*}
When Volpert product is used, 
the Rankine Hugoniot condition takes the form
\begin{equation*}
 \begin{gathered}
-s(u_{+}-u_{-})+\frac{u_{+}^2 - u_{-}^2}{2}
-(\sigma_{+}-\sigma_{-})=0\\
-s(\sigma_{+}-\sigma_{-})+\frac{u_{+}+u_{-}}{2}(\sigma_{+}-\sigma_{-})-k^2(u_{+}-u_{-})
\label{e2.8}
\end{gathered}
\end{equation*}
A $j$ - rarefaction wave is a continuous solution of \eqref{e1.1} of the form
\begin{equation*}
(u(x,t),\sigma(x,t)) = \begin{cases}
(u_{-},\sigma_{-}),&\text{if } \,\,\, x<\lambda_j(u_{-},\sigma_{-})  t,\\
(\bar{u}(x/t),\bar{\sigma}(x/t)),&\text{if} \,\,\, x< \lambda_j(u_{-},\sigma_{-}) t< x< \lambda_j(u_{+},\sigma_{+}) t,\\
 (u_{+},\sigma_{+}),&\text{if }\,\,\, x>\lambda_j(u_{+},\sigma_{+})t. 
\end{cases}
\label{e2.9}
\end{equation*}
 and  $\lambda_j(\bar{u}(\xi),\bar{\sigma}(\xi))$
is an increasing function of $\xi$. In this case  system \eqref{e1.1} can be written as,
\[ 
(A(u,\sigma)-\xi I)(\partial_\xi u,\partial_\xi \sigma)^t =0.
\]
There are two non-trivial solutions for it namely,   $\xi=\lambda_j(u,\sigma),(\partial_\xi u,\partial_\xi \sigma)^t=r_j(u,\sigma)$   for  $j=1,2$,
where, $r_1(u,\sigma)=(1,  k)^t, r_2(u,\sigma)=(1, - k)^t$. Since $\lambda_j = u+ (-1)^j k, j=1,2$, we get,
 $u(\xi)=\xi +k, \sigma(\xi)=k\xi +c$ corresponding to $\lambda_1$   and 
   $u(\xi)=\xi - k, \sigma(\xi)=- k\xi +c$ corresponding to $\lambda_2$.
Here $c$ is the arbitary constant.

In \cite{j2}, the Riemann problem was solved using Volpert
product and with Lax's admissibility for shocks. It was
shown that
corresponding to each characteristic family $\lambda_j,
j=1,2$ we can define shock waves and rarefaction waves.
Fix a state $(u_{-},\sigma_{-})$.
The set of states $(u_{+},\sigma_{+})$ which can be
connected by a single $j$- shock wave is a straight line
called
j-shock curve and is denoted by $S_j$ and the states which
can be connected by a single $j$-rarefaction wave is a
straight line is called $j$ rarefaction curve
and is denoted by $R_j$.
They are given by
\begin{equation*}
\begin{gathered}
R_1(u_{-},\sigma_{-}): \sigma=\sigma_{-}+k(u-u_{-}),
u>u_{-}\\
S_1(u_{-},\sigma_{-}): \sigma=\sigma_{-}+k(u-u_{-}),
u<u_{-}\\
R_2(u_{-},\sigma_{-}): \sigma=\sigma_{-}-k(u-u_{-}),
u>u_{-}\\
S_2(u_{-},\sigma_{-}): \sigma=\sigma_{-}-k(u-u_{-}),
u<u_{-}.\\
\label{e2.10}
\end{gathered}
\end{equation*}
Further $j$- shock speed $s_j$is given by 
\begin{equation*}
 s_j =\frac{u_{+}+u_{-}}{2} +(-1)^j k,\,\,\,j=1,2
\label{e2.11} 
\end{equation*}
The Lax entropy condition requires that the $j$- shock
satisfies inequality
\begin{equation*}
\lambda_j(u_{+},\sigma_{+}) \leq s_j \leq
\lambda_j(u_{-},\sigma_{-}).
\label{e2.12}
\end{equation*}
These curves fill in the $u-\sigma$ plane and the Riemann
problem can be solved uniquely for arbitrary initial
states
$(u_{-},\sigma_{-})$ and $(u_{_+},\sigma_{+})$
in the class of self similar functions consisting of
solutions of shock waves and rarefaction waves
separated by constant states. These constant states are
obtained from the shock curves and rarefaction curves
corresponding to the two families of the characteristic
fields.

\section{Admissible boundary values based on Riemann
problem}\label{2}
To formulate admissible boundary data, we analysed a special case  in \cite{d1},
when $u_b$ and $u_0$ lie on the $j$-Riemann invariant. There, we used vanishing viscosity approximation and the Hopf-Cole transformation.
For a more general initial and bounday data, this method does not apply.

Following \cite{jl}, we define admissible boundary values for the general Riemann boundary value problem for system \eqref{e1.1} in the following way.
\begin{equation*}
\mathcal{ER}(u_b,\sigma_b) =
\lbrace R((u_b,\sigma_b),(\bar{u},\bar{\sigma}))(\frac{x}{t} )\mid_{x=0}   \,\,\,\, : ( \bar{u},\bar{\sigma}) \in R^2  \rbrace 
\end{equation*}
where
$R((u_b,\sigma_b),(\bar{u},\bar{\sigma}))(\frac{x}{t} )$ denotes the solutions to the Riemann type initial  value problem for system \eqref{e1.1} with initial data $(u_b,\sigma_b)$ on the left and $(\bar{u},\bar{\sigma})$ on the right.

 In the following theorem we construct solution of intial boundary value problem with boundary conditions based on the Riemmann problem for
arbitrary constant initial and bounday data with no restrictions on the size of the data.

We need to consider different cases depending the position of  $(u_0,\sigma_0)$  relative to $(u_b,\sigma_b)$ . Fix  $(u_b,\sigma_b)$, and consider the wave  curves 
\begin{equation*}
\begin{gathered}
R_1(u_{b},\sigma_{b})= \{(u,\sigma): \sigma=\sigma_{b}+k(u-u_{b}),
u>u_{b}\}\\
S_1(u_{b},\sigma_{b})= \{(u,\sigma): \sigma=\sigma_{b}+k(u-u_{b}),
u<u_{b}\}\\
R_2(u_{b},\sigma_{b})=\{(u,\sigma): \sigma=\sigma_{b}-k(u-u_{b}),
u>u_{b}\}\\
S_2(u_{b},\sigma_{b})=\{(u,\sigma): \sigma=\sigma_{b}-k(u-u_{b}).
u<u_{b}\}\\
\label{e2.10}
\end{gathered}
\end{equation*}
passing through $(u_b,\sigma_b)$. These curves divide the $u-\sigma$ plane into 4 regions. 
 Let  $\Gamma_1$ denote the region between $R_1(u_b,\sigma_b)$ and  $R_2(u_b,\sigma_b)$,  $\Gamma_2$ be the region between $R_2(u_b,\sigma_b)$ and  $S_1(u_b,\sigma_b)$, $\Gamma_3$ be the region between $S_1(u_b,\sigma_b)$ and  $S_2(u_b,\sigma_b)$, $\Gamma_4$ be the region between $S_2(u_b,\sigma_b)$ and  $R_1(u_b,\sigma_b)$.
\begin{figure}[!h]
\includegraphics[width=7cm,height=7cm]{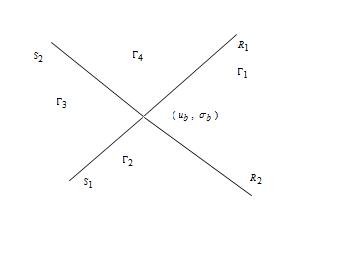}
\end{figure}

\begin{thrm}\label{6}
There exists a solution to \eqref{e1.1}  with initial condition $(u(x,0),\sigma(x,0) =( u_0,\sigma_0)$ and boundary condition   $(u(0+,t),\sigma(x+,0)) \in  \mathcal{ER}(u_b,\sigma_b)$, for any constant states  $(u_b,\sigma_b), ( u_0,\sigma_0)$. The solution consists of simple shock waves and rarefaction waves seperated by constant states.
\end{thrm}
\begin{proof}:   We need to  analyse several cases.\\
Case 1:  $(u_0,\sigma_0) \in R_1(u_b,\sigma_b)$:
It is clear that \eqref{e1.1} with\eqref{e1.2} and \eqref{e1.3} have a 1- rarefaction solution if $(u_b, \sigma_b)$ and $(u_0,\sigma_0)$ are related by 
$\sigma_0=\sigma_b + k (u_0 - u_b)$ with $u_0 > u_b$.   For $(u_b - k)t<x<(u_0 - k)t$, 
$u=\xi +k$ gives $u(x,t)=\frac{x}{t} +k$ and 
$\sigma_0= k\xi +c= k(u_0 - k) +c,\sigma_0=\sigma_b + k (u_0 - u_b)$ give $  k(u_0 - k) +c=\sigma_b + k (u_0 - u_b)$ and so $c= \sigma_b - k (u_b - k)$. So, $\sigma(x,t)= k \frac{x}{t} +  \sigma_b- k (u_b -k)$.
That is, $u(x,t)=\frac{x}{t} +k$ and $\sigma(x,t)= k \frac{x}{t} +  \sigma_b- k (u_b -k)$, for $(u_b - k)t<x<(u_0 - k)t$. Three different subcases occur. \\
 Case  (1a) :  $u_b - k>0,u_0- k>0$, \\
\begin{equation*}
(u(x,t), \sigma(x,t))=\begin{cases}
(u_b, \sigma_b) &\text{if  } 0<x<(u_b - k)t, \\
(\frac{x}{t} +k , k \frac{x}{t} +  \sigma_b - k (u_b -k)) &\text{if  } (u_b - k)t<x<(u_0 - k)t ,x>0, \\
(u_0, \sigma_0) &\text{if   } x>(u_0 - k)t .
\end{cases}
\end{equation*}\\
Case (1b) :  $u_b - k<0,u_0 - k<0,$ \\
\begin{equation*}
(u(x,t),\sigma(x,t))=(u_0,\sigma_0) .
\end{equation*} \\
Case (1c) : $u_b - k<0,u_0-k>0,$ \\
\begin{equation*}
(u(x,t), \sigma(x,t))=\begin{cases}
(\frac{x}{t} +k , k \frac{x}{t} +  \sigma_b -  k (u_b - k)) &\text{if } 0<x<(u_0 - k)t , \\
(u_0,\sigma_0) &\text{if  } x>(u_0 - k)t .
\end{cases}
\end{equation*}\\

Case 2 :  $(u_0,\sigma_0) \in R_2(u_b ,\sigma_b)$ .
 Now,\eqref{e1.1} with\eqref{e1.2} and \eqref{e1.3} have a $2$- rarefaction solution if $(u_b ,\sigma_b)$ and $(u_0,\sigma_0)$ are related by 
$\sigma_0=\sigma_b - k (u_0 - u_b)$ with $u_0 > u_b$.
For $(u_b + k)t<x<(u_0 + k)t$, 
$u=\xi - k$ gives $u(x,t)=\frac{x}{t} - k$ and 
$\sigma_0= - k\xi +c= - k(u_0 + k) +c,\sigma_0=\sigma_b - k (u_0 - u_b)$ give $  k(u_0 + k) +c=\sigma_b + k (u_0 - u_b)$ and so $c= \sigma_b - k (u_b + k)$. So, $\sigma(x,t)=   k \frac{x}{t}  +c = k \frac{x}{t} +  \sigma_b+k (u_b +k)$.
That is, $u(x,t)=\frac{x}{t} - k$ and $\sigma(x,t)= k \frac{x}{t} +  \sigma_b+ k (u_b +k)$, for $(u_b + k)t<x<(u_0 + k)t$.\\
We see corresponding solutions for the three subcases below. \\
Case (2ai) : If  $u_b+k>0,u_0+k>0,$ \\
\begin{equation*}
(u(x,t), \sigma(x,t))=\begin{cases}
(u_b, \sigma_b) &\text{if } 0<x<(u_b + k)t, \\
(\frac{x}{t} - k , - k \frac{x}{t} +  \sigma_b + k (u_b +k)) &\text{if } (u_b + k)t<x<(u_0 + k)t, \\
(u_0,\sigma_0) &\text{if } x>(u_0 +  k)t . 
\end{cases}
\end{equation*}\\
Case (2b) :  If  $u_b+k<0,u_0+k<0$,
\begin{equation*}
(u(x,t),\sigma(x,t))=(u_0,\sigma_0).
\end{equation*} \\
Case (2c) : If $u_b+k<0,u_0+k>0$, \\
\begin{equation*}
(u(x,t), \sigma(x,t))=\begin{cases}
(\frac{x}{t} + k , -  k \frac{x}{t} +  \sigma_b- k (u_b +k)) &\text{if } 0<x<(u_0 + k)t , \\
(u_0,\sigma_0) &\text{if } x>(u_0 + k)t. 
\end{cases}
\end{equation*}\\

Case 3 : $(u_0,\sigma_0) \in S_1(u_b,\sigma_b)$. Note that, \eqref{e1.1} with\eqref{e1.2} and \eqref{e1.3} have a 1- shock solution if $(u_b, \sigma_b)$ and $(u_0,\sigma_0)$ are related by 
$\sigma_0=\sigma_b + k (u_0 - u_b)$ with $u_0 < u_b$.
We then have the following two subcases, \\
Case (3a) : If $\frac{u_b +u_0}{2}- k > 0$, 
\begin{equation*}
(u(x,t), \sigma(x,t))=\begin{cases}
(u_b,\sigma_b) &\text{if } 0< x< (\frac{u_b+u_0}{2}- k )t, \\
(u_0,\sigma_0) &\text{if }  x> (\frac{u_b+u_0}{2}- k )t.
\end{cases}
\end{equation*} \\
Case (3b) : If $\frac{u_b+u_0}{2}- k < 0$, 
\begin{equation*}
(u(x,t), \sigma(x,t))= (u_0,\sigma_0). 
\end{equation*} \\

Case 4: $(u_0,\sigma_0) \in S_2(u_b,\sigma_b)$.
Now, \eqref{e1.1} with\eqref{e1.2} and \eqref{e1.3} have a 2- shock solution if $(u_b, \sigma_b)$ and $(u_0,\sigma_0)$ are related by 
$\sigma_0=\sigma_b - k (u_0 - u_b)$ with $u_0 < u_b$.
We then have the following two subcases again, \\
Case (4a) : If $\frac{u_b+u_0}{2}+ k > 0$, \\
\begin{equation*}
(u(x,t), \sigma(x,t))=\begin{cases}
(u_b ,\sigma_b) &\text{if } 0< x< (\frac{u_b +u_0}{2}+k )t, \\
(u_0,\sigma_0) &\text{if }  x> (\frac{u_b+u_0}{2} + k )t.
\end{cases}
\end{equation*} \\
Case  (4b) :  If $\frac{u_b+u_0}{2}+ k < 0$, \\
\begin{equation*}
(u(x,t), \sigma(x,t))= (u_0,\sigma_0). 
\end{equation*} \\

Next we  analyse the cases when $(u_0,\sigma_0)$ lie between the correponding  rarefaction and shock waves. 
Recall that  $\Gamma_1$ denote the region between $R_1(u_b,\sigma_b)$ and  $R_2(u_b,\sigma_b)$,  $\Gamma_2$ be the region between $R_2(u_b,\sigma_b)$ and  $S_1(u_b,\sigma_b)$, $\Gamma_3$ be the region between $S_1(u_b,\sigma_b)$ and  $S_2(u_b,\sigma_b)$, $\Gamma_4$ be the region between $S_2(u_b,\sigma_b)$ and  $R_1(u_b,\sigma_b)$. We can now consider the following four cases. \\

Case 5:   When $(u_0,\sigma_0) \in \Gamma_1(u_b,\sigma_b)$.
Here we need an intermediate state $(u_*,\sigma_*)$ such that $(u_b,\sigma_b)$ is connected to  $(u_*,\sigma_*)$  by a 1-rarefaction wave from left to right and  $(u_*,\sigma_*)$   is connected to  $(u_0,\sigma_0)$ by a 2-rarefaction wave from left to right. Now the solution to the algebraic equations for  $R_1(u_b,\sigma_b)$ and  $R_2(u_*,\sigma_*)$ give the state  
$(u_*,\sigma_*)$ namely,
\[
\sigma_{*}=\sigma_b + k (u_{*} - u_b),  \sigma_0=\sigma_{*} - k (u_0 - u_{*})
\]
which gives 
\[
u_{*} = \frac{\sigma_0 - \sigma_b}{2 k}+\frac{u_0 + u_b}{2},\,\,\,\sigma_{*} = \frac{\sigma_0 + \sigma_b}{2 }+\frac{k}{2}(u_0 + u_b).
\]
We describe the solution for each of the following 3 cases, \\
Case (5a) : $u_b-k>0,u_0+k>0,$ \\
\begin{equation*}
(u(x,t), \sigma(x,t))=\begin{cases}
(u_b,\sigma_b) &\text{if  } 0<x<(u_b - k)t, \\ 
(\frac{x}{t} +k , k \frac{x}{t} +  \sigma_b- k (u_b -k)) &\text{if } (u_b - k)t<x<(u_* - k)t ,x>0 ,\\
(u_*,\sigma_*)  &\text{if } (u_* - k)t<x<(u_* + k)t,\\
(\frac{x}{t} -k , - k \frac{x}{t} +  \sigma_*+k (u_* -k)) &\text{if } (u_* + k)t<x<(u_0 + k)t ,x>0, \\
(u_0,\sigma_0) &\text{if } x>(u_0 + k)t. 
\end{cases}
\end{equation*}\\
Case (5b) :  $u_b - k<0,u_0 + k<0,$ \\
\begin{equation*}
(u(x,t),\sigma(x,t))=(u_0,\sigma_0).
\end{equation*} \\
Case (5c) :  $u_b - k< 0,u_0 + k>0,$ \\
Subcase (5c i ) $u_* - k >0,u_* + k>0,$
\begin{equation*}
(u(x,t), \sigma(x,t))=\begin{cases}
(\frac{x}{t} +k , k \frac{x}{t} +  \sigma_b- k (u_b -k)) &\text{if } 0<x<(u_* - k)t ,x>0, \\
(u_*,\sigma_*)  &\text{if} (u_* - k)t<x<(u_* + k)t, \\
(\frac{x}{t} -k , - k \frac{x}{t} +  \sigma_*+k (u_* -k)) &\text{if } (u_* + k)t<x<(u_0 + k)t ,x>0, \\
(u_0,\sigma_0) &\text{if } x>(u_0 + k)t. 
\end{cases}
\end{equation*}\\
Subcase (5c ii) $u_* - k <0,u_* + k>0,$
\begin{equation*}
(u(x,t), \sigma(x,t))=\begin{cases}
(u_*,\sigma_*)  &\text{if } 0<x<(u_* + k)t, \\
(\frac{x}{t} -k , - k \frac{x}{t} +  \sigma_*+k (u_* -k)) &\text{if } (u_* + k)t<x<(u_0 + k)t ,x>0, \\
(u_0,\sigma_0) &\text{if } x>(u_0 + k)t. 
\end{cases}
\end{equation*}\\
Subcase (5c iii) $u_* - k <0,u_* + k<0,$
\begin{equation*}
(u(x,t), \sigma(x,t))=\begin{cases}
(\frac{x}{t} -k , - k \frac{x}{t} +  \sigma_*+k (u_* -k)) &\text{if } 0<x<(u_0 + k)t,  \\
(u_0,\sigma_0) &\text{if } x>(u_0 + k)t. 
\end{cases}
\end{equation*}\\

Case 6: When $(u_0,\sigma_0) \in \Gamma_2(u_b,\sigma_b)$.
Here we need an intermediate state $(u_*,\sigma_*)$ such that $(u_b,\sigma_b)$ is connected to  $(u_*,\sigma_*)$  by a 1-shock wave from left to right and  $(u_*,\sigma_*)$   is connected to  $(u_0,\sigma_0)$ by a 2-rarefaction wave from left to right. Now the solution to the algebraic equations for  $S_1(u_b,\sigma_b)$ and  $R_2(u_*,\sigma_*)$ give the state  $(u_*,\sigma_*)$ namely,
\[
\sigma_*=\sigma_b + k (u_* - u_b),  \sigma_0=\sigma_* - k (u_0 - u_*)
\]
which  gives
\[
 u_* = \frac{\sigma_0 - \sigma_b}{2 k}+\frac{u_0 + u_b}{2},\sigma_* = \frac{\sigma_0 + \sigma_b}{2 }+\frac{k}{2}(u_0 + u_b).
\]
We describe the solution for each of the following 3 cases,\\
Case (6a):  If $\frac{u_0 +u_b}{2} - k >0$ and  $u_0 + k>0,$ \\
\begin{equation*}
(u(x,t),\sigma(x,t))=\begin{cases}
(u_b,\sigma_b) &\text{if } 0<x< (\frac{u_0 +u_b}{2} - k)t, \\
(u_*,\sigma_*)  &\text{if } (\frac{u_0 +u_b}{2} - k)t <x<(u_* + k)t, \\
(\frac{x}{t} -k , - k \frac{x}{t} +  \sigma_*+k (u_* -k)) &\text{if } (u_* + k)t <x<(u_0 + k)t,  \\
(u_0,\sigma_0) &\text{if }  x> (u_0 + k)t. 
\end{cases}
\end{equation*}\\
Case (6b):  If $\frac{u_0 +u_b}{2} - k <0$ and  $u_0 + k<0,$\\
\[ 
(u(x,t),\sigma(x,t)) = (u_0,\sigma_0).
\]
Case(6c) :   $\frac{u_0 +u_b}{2} - k <0$ and  $u_0 + k> 0$. Here there are two subcases. \\
 Case (6c i ):  If $u_*+k>0,$ \\
\begin{equation*}
(u(x,t),\sigma(x,t))=\begin{cases}
(u_b,\sigma_b) &\text{if } 0<x<(u_* + k)t, \\
(\frac{x}{t} -k , - k \frac{x}{t} +  \sigma_*+k (u_* -k)) &\text{if } (u_* + k)t <x<(u_0 + k)t , \\
(u_0,\sigma_0) &\text{if }  x> (u_0 + k)t. 
\end{cases}
\end{equation*}\\
Case (6c ii):  If $u_*+k< 0,$ \\
\begin{equation*}
(u(x,t),\sigma(x,t))=\begin{cases}
(\frac{x}{t} -k , - k \frac{x}{t} +  \sigma_*+k (u_* -k)) &\text{if } 0 <x<(u_0 + k)t,  \\
(u_0,\sigma_0) &\text{if }  x> (u_0 + k)t. 
\end{cases}
\end{equation*}\\

Case 7:  $(u_0,\sigma_0) \in \Gamma_3(u_b,\sigma_b)$.
Here, we need an intermediate state $(u_*,\sigma_*)$ such that $(u_b,\sigma_b)$ is connected to  $(u_*,\sigma_*)$  by a 1-shock wave from left to right and  $(u_*,\sigma_*)$   is connected to  $(u_0,\sigma_0)$ by a 2-rarefaction wave from left to right. Now the solution to the algebraic equations for  $S_1(u_b,\sigma_b)$ and  $S_2(u_*,\sigma_*)$ give the state  $(u_*,\sigma_*)$ namely,
\[
\sigma_*=\sigma_b + k (u_* - u_b),  \sigma_0=\sigma_* - k (u_0 - u_*)
\]
and solving this  gives 
 \[
u_* = \frac{\sigma_0 - \sigma_b}{2 k}+\frac{u_0 + u_b}{2},\sigma_* = \frac{\sigma_0 + \sigma_b}{2 }+\frac{k}{2}(u_0 + u_b).
\] 
There are 3  subcases and we describe the solution for each of these cases.\\
 Case ( 7a):  $\frac{u_0 +u_b}{2} - k >0$ and $\frac{u_0 +u_b}{2} + k >0$. In this case, the solution is \\
\begin{equation*}
(u(x,t),\sigma(x,t))=\begin{cases}
(u_b,\sigma_b) &\text{if } 0<x< (\frac{u_0 +u_b}{2} - k)t ,\\
(u_*,\sigma_*)  &\text{if } (\frac{u_0 +u_b}{2} - k)t <x< (\frac{u_0 +u_b}{2} + k)t ,\\
(u_0,\sigma_0) &\text{if }  x>  (\frac{u_0 +u_b}{2} +k)t.
\end{cases}
\end{equation*}\\
 Case (7b):  $\frac{u_0 +u_b}{2} - k <0$ and $\frac{u_0 +u_b}{2} + k <0$. In this case, solution is 
\[
(u(x,t),\sigma(x,t))= (u_0,\sigma_0).
\]
 Case ( 7c):   $\frac{u_0 +u_b}{2} - k <0$ and $\frac{u_0 +u_b}{2} + k >0,$ \\
\begin{equation*}
(u(x,t),\sigma(x,t))=\begin{cases}
(u_*,\sigma_*)  &\text{if } 0<x< (\frac{u_0 +u_b}{2} + k)t, \\
(u_0,\sigma_0) &\text{if  }  x>  (\frac{u_0 +u_b}{2} +k)t.
\end{cases}
\end{equation*}\\

Case 8:   $(u_0,\sigma_0) \in \Gamma_4(u_b,\sigma_b)$.
Here, we need an intermediate state $(u_*,\sigma_*)$ such that $(u_b,\sigma_b)$ is connected to  $(u_*,\sigma_*)$  by a 1-rarefaction wave from left to right and  $(u_*,\sigma_*)$   is connected to  $(u_0,\sigma_0)$ by a 2-shock wave from left to right. Now the solution to the algebraic equations for  $R_1(u_b,\sigma_b)$ and  $S_2(u_*,\sigma_*)$ give the state  $(u_*,\sigma_*)$ namely,
\[\sigma_*=\sigma_b + k (u_* - u_b),  \sigma_0=\sigma_* - k (u_0 - u_*)
\]
Solving for $(u_*,\sigma_*)$  gives 
\[
u_* = \frac{\sigma_0 - \sigma_b}{2 k}+\frac{u_0 + u_B}{2},\sigma_* = \frac{\sigma_0 + \sigma_b}{2 }+\frac{k}{2}(u_0 + u_b).
\]
There are 3 subcases to consider. We describe the solution for each of the following 3 cases, as follows.\\
Case(8a):  $u_b-k>0, \frac{u_0 +u_b}{2} + k > 0,$ \\
\begin{equation*}
(u(x,t), \sigma(x,t)) =\begin{cases}
(u_b,\sigma_b) &\text{if  } 0<x<(u_b - k)t, \\ 
(\frac{x}{t} +k , k \frac{x}{t} +  \sigma_b- k (u_b -k)) &\text{if } (u_b - k)t<x<(u_* - k)t ,x>0, \\
(u_*,\sigma_*)  &\text{if } (u_* - k)t<x<(\frac{u_0 +u_b}{2} + k )t,\\
(u_0,\sigma_0) &\text{if  } x>(\frac{u_0 +u_b}{2} + k )t.
\end{cases}
\end{equation*}\\
Case (8b): $u_b - k<0,(\frac{u_0 +u_b}{2} + k ) <0,$ \\
\begin{equation*}
(u(x,t),\sigma(x,t))=(u_0,\sigma_0).
\end{equation*} \\
Case (8c) : $u_b - k< 0,(\frac{u_0 +u_b}{2} + k ) >0$ . Here there are two cases to consider.
Subcase (8c i): $u_* - k >0,$
\begin{equation*}
(u(x,t), \sigma(x,t))=\begin{cases}
(\frac{x}{t} +k , k \frac{x}{t} +  \sigma_b- k (u_b -k)) &\text{if } 0<x<(u_* - k)t,  \\
(u_*,\sigma_*)  &\text{if } (u_* - k)t<x<(\frac{u_0 +u_b}{2} + k )t, \\
(u_0,\sigma_0) &\text{if  } x>(\frac{u_0 +u_b}{2} + k )t. 
\end{cases}
\end{equation*}\\
Subcase (8c ii): $u_* - k <0,$
\begin{equation*}
(u(x,t), \sigma(x,t))=\begin{cases}
(u_*,\sigma_*)  &\text{if } 0<x<(\frac{u_0 +u_b}{2} + k )t, \\
(u_0,\sigma_0) &\text{if  } x>(\frac{u_0 +u_b}{2} + k )t .
\end{cases}
\end{equation*}\\
\end{proof}

\section {Remarks}
The vanishing viscosity analysis of our previous paper \cite{d1},  was limited to the case where $(u_0,\sigma_0)$ lie in level sets one of the Riemann invariants, that is,
$w_j(u,\sigma)=w_j(u_b,\sigma_b)$, where $w_j(u,\sigma)=\sigma-(-1)^j k u,  j=1,2$. In the present considerations, this means that  $(u_0,\sigma_0)$ lies in one of the wave curves ,
$(u_0,\sigma_0) \in R_1(u_{b},\sigma_{b})\cup S_1(u_{b},\sigma_{b})$ or $(u_0,\sigma_0) \in R_2(u_{b},\sigma_{b})\cup S_2(u_{b},\sigma_{b})$. 

We have seen in \cite{d1}, that if  $w_j(u_0,\sigma_0)=w_j(u_b,\sigma_b)$, the vanishing viscosity solution $(u,\sigma)$  with initial data $(u_0,\sigma_0)$ at $t=0$
 and boundary data $(u_b,\sigma_b)$ at $x=0$,  have the following form.

{Case 1: $u_0 -(-1)^{j+1}k =u_b -(-1)^{j+1}k >0$,} 
\[
(u(x,t),\sigma(x,t)) =
(u_0-(-1)^{j+1}k,k[-(-1)^{j+1}k+u_0]+c_j) ,\,\,\,if \,\,\,x
>0.
\]

{Case 2: $u_0 -(-1)^{j+1}k=u_b -(-1)^{j+1}k<0$,}
\[
(u(x,t),\sigma(x,t) =
(u_0-(-1)^{j+1}k,k[-(-1)^{j+1}k+u_0]+c_j).
\]

{Case 3: $0<u_b -(-1)^{j+1}k<u_0-(-1)^{j+1}k$,}
\[
(u(x,t),\sigma(x,t)) = \begin{cases} \displaystyle
{(u_b-(-1)^{j+1}k,k[-(-1)^{j+1}k+u_b]+c_j),\,\,\, if
\,\,\,x<u_b t}
\\\displaystyle
{ (x/t-(-1)^{j+1}k, k( x/t-(-1)^{j+1}k )+c_j    ),\,\,\, if \,\,\,u_b t< x<u_0
t}\\\displaystyle{(u_0- (-1)^{j+1}k,k[-(-1)^{j+1}k+u_0]+c_j),\,\,\, if \,\,\,
x>u_0 t}.
\end{cases}
\]
{Case 4:  $u_b -(-1)^{j+1}k<0<u_0-(-1)^{j+1}k$,}
\[
(u(x,t),\sigma(x,t)) = \begin{cases} \displaystyle
{(x/t-(-1)^{j+1}k, k[-(-1)^{j+1}k+ x/t]+c_j),\,\,\, if
\,\,\,0< x<u_0 t}
\\\displaystyle
{(u_0- (-1)^{j+1}k,k[u_0- (-1)^{j+1}k]+c_j),\,\,\, if \,\,\,x>u_0 t}.\end{cases}
\]
{Case 5: $u_b-(-1)^{j+1}k<0$ and $u_0-(-1)^{j+1}k\leq 0$,}
\[
(u(x,t),\sigma(x,t)) = (u_0 -(-1)^{j+1}k,k[-(-1)^{j+1}k+u_0]+c_j).
\]
{Case 6: $u_0-(-1)^{j+1}k<u_b-(-1)^{j+1}k$ and $(u_b-(-1)^{j+1}k)+(u_0 -(-1)^{j+1}k)>0$ }
\[
(u(x,t),\sigma(x,t)) = \begin{cases} \displaystyle
{(u_b- (-1)^{j+1}k,k[-(-1)^{j+1}k+u_b]+c_j),\,\,\, if
\,\,\,x<s_jt}
\\\displaystyle
{ (u_b-(-1)^{j+1}k,k[-(-1)^{j+1}k+u_b]+c_j),\,\,\, if
\,\,\,x>s_jt},
\end{cases}
\]
where $s_j=\frac{u_0 +u_b }{2}-(-1)^{j+1}k$ and $c_j=w_j(u_b,\sigma_b)$.

These formulae agree with the one constructed in theorem (2.1) for these special cases, using a different method. More over  the theorem (2.1) treats general data with no such restrictions or smallness conditions. \\


\begin{thebibliography}{10}

\bibitem{c1} J.F. Colombeau, A. Y. LeRoux, \emph{Multiplications
of distributions in elasticity
and hydrodynamics}, J. Math. Phys. 29 , no. 2,
315-319 (1988) .

\bibitem{c2} J. J. Cauret, J. F. Colombeau and A.-Y.
LeRoux,
\emph{Discontinous generalized solutions of nonlinear
nonconservative
hyperbolic equation}, J. Math. Anal. Appl. {\bf139}
, 552--573 (1989).

\bibitem {d1} Kayyunnapara  Divya Joseph and  P. A  Dinesh, \emph{Initial boundary value problem for a nonconservative
system in elastodynamics}, Acta Mathematica Scientia, Vol. 38, Issue 3, 1043-1056, (2018). 

\bibitem{j2}
K. T. Joseph and P. L. Sachdev,
\emph{Exact solutions for some nonconservative
Hyperbolic Systems}, Int. J. Nonlinear Mech. {\bf 38}
 1377--1386, (2003).

\bibitem{jl}
K.T.Joseph and  P. G. LeFloch, \emph{Boundary layers in weak solution to hyperbolic conservation laws},  Arch. Rat. Mech. Anal. {\bf 147}  47-88, (1999).
\end{thebibliography}
\end{document}